\documentclass[12pt,twoside,leqno]{amsart}
\usepackage[all]{xy}
\usepackage{graphicx}
\usepackage{amsmath}
\usepackage{color}
\usepackage{listings}

\usepackage{url}

\numberwithin{equation}{section}

\title[Toric Fano manifolds with positive second Chern characters]
{Toric Fano manifolds of dimension at most eight with positive second Chern 
characters}
\author{Yuji Sano, Hiroshi Sato and Yusuke Suyama} 
\date{2020/8/27, version 0.08}
\address{Department of Applied Mathematics, Faculty of Sciences, 
Fukuoka University, 8-19-1, Nanakuma, Jonan-ku, Fukuoka 814-0180, Japan}
\email{sanoyuji@fukuoka-u.ac.jp}
\address{Department of Applied Mathematics, Faculty of Sciences, 
Fukuoka University, 8-19-1, Nanakuma, Jonan-ku, Fukuoka 814-0180, Japan}
\email{hirosato@fukuoka-u.ac.jp}
\address{Department of Mathematics, Graduate School of 
Science, Osaka University, Toyonaka, Osaka 560-0043, Japan}
\email{y-suyama@cr.math.sci.osaka-u.ac.jp}

\makeatletter
    
    \@addtoreset{equation}{section}
\makeatother

\newcommand{\Pic}[0]{{\operatorname{Pic}}}

\newcommand{\G}[0]{{\operatorname{G}}}

\newtheorem{thm}{Theorem}[section]

\newtheorem{prop}[thm]{Proposition}

\newtheorem{conj}[thm]{Conjecture}

\theoremstyle{definition}

\newtheorem{defn}[thm]{Definition}

\newtheorem{rem}[thm]{Remark}
\newtheorem*{ack}{Acknowledgments}       

\begin{document}
\bibliographystyle{amsalpha+}

\begin{abstract}
We show that any toric Fano manifold of dimension at most eight 
with the positive second Chern character is isomorphic to 
the projective space by using \texttt{polymake}. 
\end{abstract}

\maketitle

\tableofcontents
\section{Introduction} 
\thispagestyle{empty}
Smooth Fano varieties are very important objects in algebraic geometry, 
though the definition is very simple, that is, a smooth projective variety with 
ample anti-canonical divisor. 
By Nakai-Moishezon criterion, this condition implies that the intersection $\mathrm{ch}_{1}(X) \cdot C$ is positive for any curve $C$ on $X$, where $\mathrm{ch}_{1}(X)$ is  the first Chern character of $X$.
Replacing the first Chern character (resp. a curve $C$) by the second Chern character (resp. a surface $S$), the following notion is introduced.

\begin{defn}
A smooth projective variety $X$ over an algebraically closed field $k=\overline{k}$ 
is said to be $\mathrm{ch}_2$-{\em positive} (resp. $\mathrm{ch}_2$-{\em nef}) if 
\begin{equation*}\label{eq:ch2positive}
	\mathrm{ch}_2(X)\cdot S>0\quad (\mbox{resp. }\mathrm{ch}_2(X)\cdot S\ge 0)
\end{equation*}
for any subsurface $S\subset X$, where 
$\mathrm{ch}_2(X)=\frac{1}{2}(\mathrm{c}_1^2-2\mathrm{c}_2)$ is 
the second Chern character of $X$. 
\end{defn}

\noindent
$\mathrm{ch}_2$-nef Fano manifolds are first studied by de Jong and Starr \cite{ds} 
in connection with the existence of rational surfaces on Fano manifolds. 
However, 
only few examples of $\mathrm{ch}_2$-positive manifolds are known.
For instance, the known examples of $\mathrm{ch}_2$-positive smooth projective toric varieties (not necessarily Fano) are only projective spaces (see \cite{sato3} and \cite{satosuyama2}) at the moment.
In this paper, we restrict $X$ to be a toric \textit{Fano} manifold.
Nobili \cite{nobili} and the second author \cite{sato2} proved that 
any $\mathrm{ch}_2$-positive smooth toric Fano $4$-fold is 
isomorphic to $\mathbb{P}^4$. 
The main result of this paper is to classify  $\mathrm{ch}_2$-positive smooth toric Fano $d$-folds for $5\le d\le 8$.
The result is similar as the known results.
\begin{thm}\label{thm:main}
Let $X$ be a smooth toric Fano $d$-fold. If $X$ is $\mathrm{ch}_2$-positive and $d\le 8$, 
then $X$ is isomorphic to the $d$-dimensional projective space $\mathbb{P}^d$. 
\end{thm}
\noindent
Our classification is owed to the database of smooth reflexive polytopes given by \O bro \cite{obro} and Paffenholz \cite{paffenholz}, and the software called \texttt{polymake} \cite{polymake} for computations relevant to polytopes.

\medskip

This paper is organized as follows: 
In Section \ref{junbi}, we recall the formula to compute the intersection number of 
$\mathrm{ch}_2(X)$ and a torus-invariant subsurface $S$ on $X$ whose Picard number is equal to two.
This formula is implemented as a script in \texttt{polymake} in Section \ref{mainsec}.
Section \ref{pseudosec} is devoted to the calculations of the intersection numbers on 
so-called pseudo-symmetric toric Fano varieties $\widetilde{V}^d$ and $V^d$. 
One of them is the exceptional case we cannot apply the script to.
In Section \ref{mainsec}, we conclude the main result of this paper with the script.

\begin{ack}
The first author was partly supported by JSPS KAKENHI 
Grant Number JP17K05233. 
The second author was partly supported by JSPS KAKENHI 
Grant Number JP18K03262. 
The third author was partly supported 
by JSPS KAKENHI 
Grant Number JP18J00022. 
\end{ack}

\section{The intersection $\mathrm{ch}_{2}(X)\cdot S$ for a subsurface $S$ of Picard number two}\label{junbi} 
First, we collect some basic facts of toric geometry which we need. 
For details, see \cite{cls}, \cite{fujino-sato}, \cite{fulton}, \cite{matsuki}, 
\cite{oda} and \cite{reid}. 

Let $X=X_\Sigma$ be the smooth projective toric $d$-fold over an 
algebraically closed field $k=\overline{k}$ associated to a fan $\Sigma$ in 
$N:=\mathbb{Z}^d$. For $\{v_1,\ldots,v_l\}\subset N$, 
$\langle v_1,\ldots,v_l\rangle$ stands for the cone in $N_\mathbb{R}:=N\otimes\mathbb{R}$ 
generated by $v_1,\ldots,v_l$. Let $\G(\Sigma)$ be the set of 
primitive generators of one-dimensional cones in $\Sigma$. It is well-known that 
\[
\mathrm{ch}_2(X)=\frac{1}{2}\sum_{x\in\G(\Sigma)}D_x^2,
\]
where $D_x$ is the torus-invariant prime divisor corresponding to $x\in\G(\Sigma)$. 

For a smooth projective toric variety $X$ and 
a torus-invariant subsurface $S\subset X$ of Picard number one, it is 
well-known that the inequality 
$\mathrm{ch}_2(X)\cdot S>0$ always holds. On the other hand, 
the intersection number of $\mathrm{ch}_2(X)$ and 
any torus-invariant subsurface of Picard number two can be easily 
calculated as follows: 
Let $X=X_\Sigma$ be a smooth projective toric 
$d$-fold, and $S\subset X$ a torus-invariant subsurface 
of Picard number two. 
Let $\tau\in\Sigma$ be the $(d-2)$-dimensional cone 
associated to $S$ and 
$\tau\cap\G(\Sigma)=\{x_1,\ldots,x_{d-2}\}$. 
There exist exactly four maximal cones 
\[
\tau+\langle y_1,y_3\rangle,\ 
\tau+\langle y_2,y_3\rangle,\ 
\tau+\langle y_1,y_4\rangle\mbox{ and }
\tau+\langle y_2,y_4\rangle
\] 
in $\Sigma$, where $\{y_1,y_2,y_3,y_4\}\subset\G(\Sigma)$. 
Let 
\[
y_1+y_2+c_3y_3+a_1x_1+\cdots+a_{d-2}x_{d-2}=0
\mbox{ and }
\]
\[
y_3+y_4+c_1y_1+e_1x_1+\cdots+e_{d-2}x_{d-2}=0
\]
be the wall relations corresponding to 
$(d-1)$-dimensional cones 
$\tau+\langle y_3\rangle$ and $\tau+\langle y_1\rangle$, 
respectively, where $a_1,\ldots,a_{d-2},
c_1,c_3,e_1,\ldots,e_{d-2}\in\mathbb{Z}$. 
Then the following holds:
\begin{prop}[{\cite[Proposition 3.6]{satosumi}}]\label{intersection}
\[
2\mathrm{ch}_2(X)\cdot S=
-c_1\left(2+c_3^2+a_1^2+\cdots+a_{d-2}^2\right)
\]
\[
+2\left(c_1+c_3+a_1e_1+\cdots+a_{d-2}e_{d-2}\right)
-c_3\left(2+c_1^2+e_1^2+\cdots+e_{d-2}^2\right).
\]
\end{prop}
\noindent
This formula is implemented as a script explained in Section \ref{mainsec}.

\section{Pseudo-symmetric toric Fano manifolds}\label{pseudosec} 

In this section, we show the non-positivity of $\mathrm{ch}_2(\widetilde{V}^{d})$ and 
$\mathrm{ch}_2(V^{d})$, where $\widetilde{V}^d$ and $V^d$ are so-called 
{\em pseudo-symmetric} toric Fano varieties studied in \cite{ewald} and \cite{voskre1}. 
This result complements our script in Section \ref{mainsec}. 

\medskip

For $d=2n\in 2\mathbb{N}$, we define 
the $d$-dimensional smooth toric Fano varieties $\widetilde{V}^d$ and $V^d$ as follows 
(for the precise description of these varieties, please see \cite{casagrande1}): 
Let $\{e_1,\ldots,e_{2n}\}\subset N_\mathbb{R}$ be the standard basis, 
and put 
\[
x_1:=e_1,\ldots,x_{2n}:=e_{2n},x_{2n+1}:=-(e_1+\cdots+e_{2n}),
\]
\[
y_1:=-e_1,\ldots,y_{2n}:=-e_{2n},y_{2n+1}:=e_1+\cdots+e_{2n}.
\]
Then $\widetilde{V}^{d}$ is the smooth toric Fano $d$-fold $X_{\widetilde{\Sigma}}$ such that
\[
\G(\widetilde{\Sigma})=\{x_1,\ldots,x_{2n+1},y_1,\ldots,y_{2n}\}, 
\] 
while $V^{d}$ is the smooth toric Fano $d$-fold $X_{\Sigma}$ such that
\[
\G(\Sigma)=\{x_1,\ldots,x_{2n+1},y_1,\ldots,y_{2n+1}\}.
\] 

\medskip

$\widetilde{V}^{2}$ and $V^2$ are isomorphic to the del Pezzo surfaces of 
degree $7$ and $6$, respectively, which are not $\mathrm{ch}_2$-positive. 
Hence we may assume $d \geq 4$. 

\begin{thm}\label{voskre}
$\widetilde{V}^{d}$ and $V^{d}$ are {\em not} $\mathrm{ch}_2$-positive 
for any $d=2n\in2\mathbb{N}$. 
\end{thm}

\begin{proof}
For $\widetilde{V}^{d}$, the Picard number of 
the torus-invariant surface $S_\tau$ 
associated to the $(d-2)$-dimensional cone 
\[
\tau:=\langle x_1,\ldots,x_{d-2}\rangle\in\widetilde{\Sigma}
\]
is two, because there exist exactly four maximal cones
\[
\tau+\langle x_{d-1},x_d\rangle,\ 
\tau+\langle x_{d-1},y_d\rangle,\ 
\tau+\langle y_{d-1},x_d\rangle\mbox{ and }
\tau+\langle y_{d-1},y_d\rangle
\]
which contain $\tau$ as a face. 
The relations 
\[
x_{d-1}+y_{d-1}=0\mbox{ and }x_d+y_d=0
\]
tell us that $S_\tau\cong\mathbb{P}^1\times\mathbb{P}^1$, and 
$\mathrm{ch}_2(\widetilde{V}^{d})\cdot S_\tau=0$ 
by Proposition \ref{intersection}. 
Therefore, $\widetilde{V}^{d}$ is {\em not} $\mathrm{ch}_2$-positive. 

\medskip

For $V^d$, there are no torus-invariant subsurfaces of Picard number two in $V^d$. 
So, we cannot apply Proposition \ref{intersection}. 
In this case, we can show the non-positivity of $\mathrm{ch}_2(V^d)$ 
by using the typical method of the calculation of intersection numbers: 
It is well-known that the maximal cones of $\Sigma$ are 
\[
\langle x_{i_1},\ldots,x_{i_n},y_{j_1},\ldots,y_{j_n} \rangle,
\]
where $1\le i_1<\cdots<i_n\le 2n+1$, $1\le j_1<\cdots<j_n\le 2n+1$ and 
$\{i_1,\ldots,i_n\}\cap\{j_1,\ldots,j_n\}=\emptyset$. 
Let $S_\tau\subset V^{d}$ be the torus-invariant subsurface 
associated to the $(2n-2)$-dimensional cone 
\[
\tau:=\langle x_1,\ldots,x_{n-1},y_{n},\ldots,y_{2n-2} \rangle.
\]
There exist exactly six maximal cones 
\[
\tau+\langle x_{2n-1},y_{2n} \rangle,\ 
\tau+\langle x_{2n-1},y_{2n+1} \rangle,\ 
\tau+\langle y_{2n-1},x_{2n} \rangle,\ 
\tau+\langle y_{2n-1},x_{2n+1} \rangle,
\]
\[
\tau+\langle x_{2n},y_{2n+1} \rangle
\mbox{ and }
\tau+\langle y_{2n},x_{2n+1} \rangle
\]
which contain $\tau$. 
Namely, $S_\tau$ is isomorphic to the del Pezzo surface $S_6$ of degree $6$. 
For $1\le i\le 2n+1$, let $D_i$ and $E_i$ be the torus-invariant prime divisors 
corresponding to $x_i$ and $y_i$, respectively. Then, we have 
relations
\[
D_i-E_i-D_{2n+1}+E_{2n+1}=0\quad (1\le i\le 2n)
\]
in $\Pic (V^{d})$. Obviously, 
\[
E_1^2\cdot S_\tau=\cdots =E_{n-1}^2\cdot S_\tau=
D_{n}^2\cdot S_\tau=\cdots=D_{2n-2}^2\cdot S_\tau=0\mbox{ and }
\]
\[
D_{2n-1}^2\cdot S_\tau=D_{2n}^2\cdot S_\tau=D_{2n+1}^2\cdot S_\tau=
E_{2n-1}^2\cdot S_\tau=E_{2n}^2\cdot S_\tau=E_{2n+1}^2\cdot S_\tau=-1.
\]
On the other hand, 
\[
D_1^2\cdot S_\tau=
(E_1+D_{2n+1}-E_{2n+1})\cdot (E_1+D_{2n+1}-E_{2n+1}) \cdot S_\tau
\]
\[
=(E_1^2+D_{2n+1}^2+E_{2n+1}^2+2E_1D_{2n+1}-2D_{2n+1}E_{2n+1}-2E_1E_{2n+1})
\cdot S_\tau
\]
\[
=0+(-1)+(-1)+2-2-2\times 0=-2.
\]
By symmetry, we have 
\[
D_i^2\cdot S_\tau=-2\mbox{ for }1\le i\le n-1,\mbox{ while }
E_j^2\cdot S_\tau=-2\mbox{ for }n\le j\le 2n-2.
\] 
Therefore, since 
\[
2\mathrm{ch}_2(V^{d})\cdot S_\tau=-6+(2n-2)\times (-2)<0, 
\]
$V^{d}$ is {\em not} $\mathrm{ch}_2$-positive. 
\end{proof}

\section{Main results}\label{mainsec} 

In this section, we give a proof of Theorem \ref{thm:main}.
Our proof consists of three ingredients; the database of smooth reflexive polytopes, a script  to compute the intersection $\mathrm{ch}_{2}(X)\cdot S$ for a subsurface $S$ with Picard number two, and Theorem \ref{voskre}.

\subsection{The database of smooth reflexive polytopes}
Our classification is owed to the database of smooth reflexive lattice polytopes.
\O bro \cite{obro}  provided the algorithm to determine all smooth toric Fano $d$-folds for any $d\in\mathbb{N}$.
By using his algorithm, \O bro classified all smooth toric Fano $d$-folds for  $d \le 8$.
As for $d=9$, the classification was done by an improved implementation of the algorithm by B.~Lorentz and A.~Paffenholz \cite{paffenholz}.
As a result,  the numbers of the isomorphism classes of smooth toric Fano $d$-folds for $d\le 9$ are given as follows.
\begin{table}[h]
\begin{center}
\begin{tabular}{|c||c|c|c|c|c|c|c|c|c|}
\hline
$d$ & $1$ & $2$ & $3$ & $4$ & $5$ & $6$ & $7$ & $8$ & $9$  \\ \hline\hline
$\#$ of toric Fano $d$-folds & $1$ & $5$ & $18$ & $124$ & $866$ & $7622$ & 
$72256$ & $749892$ & $8229721$ \\ \hline
\end{tabular}
\end{center}
\end{table}\\
The data of smooth toric Fano varieties for dimensions $3$ to $9$ is given in \texttt{polymake} format on the web:
\begin{center}
\url{https://polymake.org/polytopes/paffenholz/www/fano.html}
\end{center}
We use the files named \texttt{fano-v}$k$\texttt{d.tgz} ($3 \le k \le 6$), \texttt{fano-v}$7$\texttt{d-}$\ell$\texttt{.tgz} ($0\le \ell \le 7$) and \texttt{fano-v}$8$\texttt{d-}$\ell$\texttt{.tgz} ($0\le \ell \le 74$) on the above webpage. 

\subsection{Implementation}\label{subsec:Implementation}
We implement Proposition \ref{intersection} as follows.
\begin{enumerate}
\item Obtain a list of all primitive generators of the fan $\Sigma$ of each smooth toric Fano $d$-fold from the files of the database.
\item Obtain a list of primitive generators consisting of each $(d-2)$-dimensional cone $\tau$ in $\Sigma$, then enumerate maximal cones containing $\tau$.
\item If the number of maximal cones containing $\tau$ is equal to four, then take one of them as $\sigma_{1}$.
In addition, obtain the generators of $\tau$ as $x_{1}, \ldots, x_{d-2}$, and the generators of $\sigma_{1}$ except $x_{1}, \ldots, x_{d-2}$ as $y_{1}, y_{3}$.
\item Obtain a maximal cone $\sigma_{2}$ such that $\sigma_{2}$ contains the $(d-1)$-dimensional cone $\tau +\langle y_3\rangle$ but does not contain $y_{1}$ as a generator. 
Then, we get the generator of $\sigma_{2}$ except $\{x_{1}, \ldots, x_{d-2}, y_{3}\}$ as $y_{2}$.
\item Obtain a maximal cone $\sigma_{3}$ such that $\sigma_{3}$ contains the $(d-1)$-dimensional cone $\tau +\langle y_1\rangle$ but does not contain $y_{3}$ as a generator. 
Then, we get the generator of $\sigma_{3}$ except $\{x_{1}, \ldots, x_{d-2}, y_{1}\}$ as $y_{4}$.
\item Compute the coefficients $a_1,\ldots,a_{d-2}, c_1,c_3,e_1,\ldots,e_{d-2}$ in the wall relations. 
Then, compute the intersection $\mathrm{ch}_2(X)\cdot S$ where $S$ is the subsurface corresponding to the cone $\tau$ by substituting $a_1,\ldots,a_{d-2}, c_1,c_3,e_1,\ldots,e_{d-2}$ into the formula in Proposition \ref{intersection}.
\end{enumerate}

Let us see the above implementation in each step.
One may consult the website (\url{https://polymake.org/doku.php}) for the installation of \texttt{Polymake}.
Download the files named \texttt{fano-v*d.tgz} of the database of smooth reflexive polytopes from the website as noted before, then put them on any directory.
Our script is written in Perl which is an interface language of \texttt{Polymake}.

\noindent
\\
\textbf{Step (1)}
We use the application \texttt{fan} to compute calculations on a fan.
The function \texttt{unpack\_tarball} in the script \texttt{tarball} restores the files \texttt{fano-v*d.tgz}.
We substitute it into the array \texttt{@a}.
We extract a data of a smooth reflexive polytope from \texttt{@a} and substitute it into \texttt{@Q}.
The function \texttt{polarize} induces the polar dual polytope \texttt{\$P} to \texttt{\$Q}.
The function \texttt{face\_fan} converts the polytope \texttt{\$P} into the data of the fan (named  \texttt{\$fan}).
We extract the set of generators of \texttt{\$fan} by the function \texttt{RAYS} as an array \texttt{\$rays}.

\noindent
\\
\textbf{Step (2)}
The function \texttt{MAXIMAL\_CONES} is applied to \texttt{\$rays} and returns the family of the labelled set of indices of generators in \texttt{\$rays} generating a maximal cone.
The function \texttt{N\_MAXIMAL\_CONES} returns the number of maximal cones in \texttt{\$fan}.
The function \texttt{CONES->[k]} returns the family of the labelled set of indices of generators in \texttt{\$rays} generating a (\texttt{k}$+1$)-cone.
The function \texttt{incl} is to analyze the inclusion relation of given two sets. 
The value \texttt{incl (A,B)} is equal to one if \texttt{A}  contains \texttt{B}.
Hence, the value \texttt{\$link} is equal to the number of maximal cones containing a (\texttt{\$d}$-2$)-cone \texttt{\$fan->CONES->[\$d-3]->row(\$c0)} where \texttt{\$c0} is a loop counter to indicate a (\texttt{\$d}$-2$)-cone in \texttt{\$fan->CONES->[\$d-3]}.

\noindent
\\
\textbf{Step (3)}
If \texttt{\$link} is equal to four at \texttt{\$c0}, the corresponding subsurface $S$ has the Picard number two.
Then, we substitute their generators into an array \texttt{@X}.
Here an element in \texttt{@X} denotes the vectors $x_{i}$ as in Proposition \ref{intersection}.
In addition, since generators in \texttt{\$rays} are not necessarily primitive, we need to convert them to be primitive by the function \texttt{primitive}.
By using \texttt{incl} as in Step (2), we obtain a maximal cone \texttt{\$fan->MAXIMAL\_CONES->[\$c1]} containing \texttt{\$fan->CONES->[\$d-3]->row(\$c0)}. 
Taking the difference between \texttt{\$fan->MAXIMAL\_CONES->[\$c1]} and \texttt{\$fan->CONES->[\$d-3]->row(\$c0)}, we obtain the set \texttt{\$u} of the indices of the generator rays corresponding to $y_{1}, y_{3}$.
Then, we obtain the vectors \texttt{\$Y[0], \$Y[1]} corresponding to $y_{1}, y_{3}$ and their indices \texttt{\$y1, \$y3}.

\noindent
\\
\textbf{Step (4) and (5)} 
Taking the set of generators of $\tau +\langle y_3\rangle$ by \texttt{\$fan->CONES->[\$d-3]->row(\$c0) +\$y3}, we repeat a similar procedure as Step (3). 


\noindent
\\
\textbf{Step (6)}
First, we compute the coefficients $c_{3}, a_{1}, \ldots, a_{d-2}$ in the former of the two wall relations in Proposition \ref{intersection}.
Substituting \texttt{\$Y[0]}, \texttt{\$Y[1]} and \texttt{@X} into the array \texttt{@M}, we convert \texttt{@M} into a $d\times d$-matrix \texttt{\$mat}.
Then, we compute the coefficients as \texttt{\$coef1} by using the function \texttt{cramer (A,b)} which gives the solution of the system $A\mathbf{x}=\mathbf{b}$ by Cramer's rule.
Remark \texttt{\$coef1->[0]} is always equal to one, which corresponds to the coefficient of $y_{1}$ in the former of the two wall relations.
Moreover \texttt{\$coef1->[1]} corresponds to $c_{3}$, and \texttt{\$coef1->[k]} ($2\le$\texttt{k}$\le$\texttt{\$d-1}) corresponds to $a_{k-1}$ respectively.
Similarly, we obtain the coefficients $c_{1}, e_{1}, \ldots, e_{d-2}$ in the latter of the two wall relations in Proposition \ref{intersection} as \texttt{\$coef2}.
Substituting \texttt{\$coef1} and \texttt{\$coef2} into the formula in Proposition \ref{intersection}, we obtain the intersection $2\mathrm{ch_{2}}(X) \cdot S$ as \texttt{\$intersection}.
\\

See a practical script to determine whether $X$ is $\mathrm{ch}_2$-positive or not in the last of this section.

\subsection{Results and conjectures}
By using our script, we find the following results.
\begin{prop}\label{prop:d=odd}
For any smooth toric Fano $d$-fold $X$ of $d=5,\,7$ and $\rho(X) \ge 2$, there exists a torus-invariant surface $S\subset X$ such that $\rho(S)=2$ and $\mathrm{ch}_2(X)\cdot S\le 0$.
In particular, $X$ is not $\mathrm{ch}_2$-positive.
\end{prop}
As for $d=4,\,6,\,8$, there exist the exceptional cases we cannot apply our script to.
\begin{prop}\label{prop:d=even}
For any smooth toric Fano $d$-fold $X$ of $d=4,\,6,\, 8$ and $\rho(X) \ge 2$ except for $V^d$, there exists a torus-invariant surface $S\subset X$ such that $\rho(S)=2$ and $\mathrm{ch}_2(X)\cdot S\le 0$.
\end{prop}
Combining this proposition with Theorem \ref{voskre}, it is proved that any smooth toric Fano $d$-fold $X$ of $d=4,\,6,\, 8$ and $\rho(X) \ge 2$ is not $\mathrm{ch}_2$-positive.
With Proposition \ref{prop:d=odd}, Proposition \ref{prop:d=even} and the known results for $d=1,2,3$, we can conclude Theorem \ref{thm:main}.

The lists of our main results are available on the web:
\begin{center}
\url{https://sites.google.com/a/fukuoka-u.ac.jp/satoric/toricfano_ch2}
\end{center}
In the lists, we explicitly describe a surface $S\subset X$ such that 
$\mathrm{ch}_2(X)\cdot S\le 0$ for any smooth toric Fano $d$-fold 
$X$ of $5\le d\le 7$ 
except for $\mathbb{P}^5$, $\mathbb{P}^6$, $V^6$ and $\mathbb{P}^7$. 

Thus, we end this subsection by proposing the following two conjectures:
\begin{conj}\label{conj1}
Let $X$ be a smooth toric Fano $d$-fold. If $X$ is $\mathrm{ch}_2$-positive, 
then $X$ is isomorphic to the $d$-dimensional projective space $\mathbb{P}^d$. 
\end{conj}
\begin{conj}\label{conj2}
Let $X$ be a smooth toric Fano $d$-fold. If $X$ is 
isomorphic to neither $\mathbb{P}^d$ nor $V^d$, 
then there exists a torus invariant subsurface $S\subset X$ 
of Picard number two such that 
$\mathrm{ch}_2(X)\cdot S\le 0$. 
\end{conj}
\begin{rem}
Obviously, Conjecture \ref{conj2} implies Conjecture \ref{conj1} by Theorem \ref{voskre}.
\end{rem}

\subsection{Script}
In this subsection, we build the scripts explained in Subsection \ref{subsec:Implementation} into  a practical script to determine whether $X$ is $\mathrm{ch}_2$-positive or not.
The following script returns a message ``\texttt{not ch\_2-positive}'' if $X$ admits a surface $S\subset X$ such that $\rho(S)=2$ and $\mathrm{ch}_2(X)\cdot S\le 0$.
\begin{lstlisting}[basicstyle=\ttfamily\footnotesize, frame=single]
use strict;
use warnings;
use utf8;
use application "fan";
binmode STDIN, ':encoding(cp932)';
binmode STDOUT, ':encoding(cp932)';
binmode STDERR, ':encoding(cp932)';

&ch2positive("fano-v*d.tgz"); #enter the full path of the file.

sub ch2positive {
    script("tarballs"); 
    my @a=unpack_tarball($_[0]);  
    my $d=*; 	#input the dimension of manifolds
    my $c0=0;	#loop counter
    my $c1=0;	#loop counter
    my $poly_c=0;	#loop counter for polytopes

    my $link=0;	#number of maximal cones 
    				containing a cone of codimension two
    my @X = cols(zero_matrix($d-2, $d)); #array for the vectors x_i
    my @Y = cols(zero_matrix(4, $d));	#array for the vector y_i
    my $y1; #index of vector y_1 in $v1
    my $y2; #index of vector y_2 in $v1
    my $y3; #index of vector y_3 in $v1
    my $y4; #index of vector y_4 in $v1
    my $coef1; #coefficients in the first wall relation
    my $coef2; #coefficients in the second wall relation
    my $square1=0; 
    	#square sum of x_i in the first wall relation
    my $square2=0;
    	#square sum of the coefficients in the second wall relation
    my $cross=0; #inner product of $coef1 and $coef2
    my $intersection=0; #intersection number of ch_2(X) and S
    my $pc0=0; #counter for the number of surfaces 
			with non-positive intersection number

    my $v0; #vector whose elements are indices of 1-cones
    my $v1; #vector whose elements are indices of 1-cones
    my $u0; #set of indices of 1-cones
    my $set; #set of indices of 1-cones
    my @M = cols(zero_matrix($d,$d)); #matrix consisting of @X and @Y

#Step (1)

    while($poly_c < $#a+1){		
    	   print $_[0];
	   print "_";
	   print $poly_c;
	   print "\n";
 	   my $Q = $a[$poly_c];  
           my $P = polarize($Q);    
	   my $fan = face_fan($P);   
	   my $rays = $fan->RAYS;
	   my $max_cones = $fan->MAXIMAL_CONES;
           my $N_max_cones = $fan->N_MAXIMAL_CONES;
  
#Step (2)

	  $c0=0;
   	  $pc0=0;
          while ($c0 < $fan->F_VECTOR->[$d-3]){
          	$c1=0;
        	$link=0;
        	my $ind_surface = $fan->CONES->[$d-3]->row($c0);
  	        while ($c1 < $N_max_cones){	
            		if (incl($max_cones->[$c1], $ind_surface)==1){
                		$link++;
            		}
            		$c1++;
        	}

        if ($link==4){
            $square1 =0;
            $square2 =0;
            $cross =0;

            $c1=0;
            $v0 = new Vector<Int>($ind_surface);
            while ($c1<$d-2){
                $X[$c1]= new Vector (primitive($rays->[$v0->[$c1]]));
                $c1++;
            }

#Step (3)

            $c1=0;
            while ($c1<$N_max_cones){		
                if (incl($max_cones->[$c1], $ind_surface)==1){
                    $u0 = $max_cones->[$c1] - $ind_surface; 
                    $v1 = new Vector<Int>($u0);
                    $Y[0] = new Vector(primitive($rays->[$v1->[0]]));
                    $Y[1] = new Vector(primitive($rays->[$v1->[1]]));
                    $y1 = $v1->[0];
                    $y3 = $v1->[1];
                    $c1=$N_max_cones;
                } else{
                        $c1++;
                }
            }               

#Step (4)

            $c1=0;
            $set = $ind_surface +$y3;
            while ($c1<$N_max_cones){		
                if (incl($max_cones->[$c1], $set)==1){
                    $u0 = $max_cones->[$c1] - $set;
                    $v1 = new Vector<Int>($u0);
                    if ($v1->[0]!=$y1){
                        $Y[2] = new Vector(primitive($rays->[$v1->[0]]));
                        $y2 = $v1->[0];
                        $c1=$N_max_cones;
                    } else {
                        $c1++;
                    }                    
                } else {
                    $c1++;
                }
            }
            
#Step (5)

           $c1=0;
            $set = $ind_surface +$y1;
            while ($c1<$N_max_cones){		
                if (incl($max_cones->[$c1], $set)==1){
                    $u0 = $max_cones->[$c1] - $set;
                    $v1 = new Vector<Int>($u0);
                    if ($v1->[0]!=$y3){
                        $Y[3] = new Vector(primitive($rays->[$v1->[0]]));
                        $y4 = $v1->[0];
                        $c1=$N_max_cones;
                    } else {
                        $c1++;
                    }                    
                } else {
                    $c1++;
                }
            }
            
#Step(6)
            
            $M[0]=$Y[0];
            $M[1]=$Y[1];
            $c1=2;
            while($c1<$d){
                $M[$c1] = $X[$c1-2];
                $c1++;
            }
            my $mat= transpose(new Matrix (@M));
            $coef1=cramer($mat,(-1)*$Y[2]);

            $M[0]=$Y[1];
            $M[1]=$Y[0];
            $mat= transpose(new Matrix (@M));
            $coef2=cramer($mat,(-1)*$Y[3]);

            $c1=2;
            while ($c1<$d){
                $square1 += ($coef1->[$c1])*($coef1->[$c1]);
                $square2 += ($coef2->[$c1])*($coef2->[$c1]);
                $cross += ($coef1->[$c1])*($coef2->[$c1]);
                $c1++;
            }

            $intersection = 
            - $coef2->[1]*(2+$coef1->[1]*$coef1->[1]+$square1)
            +2*($coef1->[1]+$coef2->[1]+$cross)
            -$coef1->[1]*(2+$coef2->[1]*$coef2->[1]+$square2);

            if ($intersection<=0){
                $pc0++;
                $c0 = $fan->F_VECTOR->[$d-3];
                last;
            }
        }
        $c0++;
    }
    if ($pc0!=0){
        print "not ch_2-positive";
        print "\n\n";
    } else {
        print "cannot determine via surfaces of Picard number two";
        print "\n\n";
    }
    $poly_c++;
    }
}
\end{lstlisting}
\begin{rem}
The function \texttt{F\_VECTOR->[k]} returns the number of (\texttt{k}$+1$)-dimensional cones in a given fan.
\end{rem}


\end{document}